\documentclass[12pt,twoside]{amsart}
\usepackage{amsmath, amsthm, amscd, amsfonts, amssymb, graphicx}
\usepackage{enumerate}
\usepackage[colorlinks=true,
linkcolor=blue,
urlcolor=cyan,
citecolor=red]{hyperref}
\usepackage{mathrsfs}
\newtheorem{theorem}{Theorem}[section]
\newtheorem{lemma}{Lemma}[section]
\newtheorem{remark}{Remark}[section]

\newtheorem{corollary}{Corollary}[section]

\newtheorem{proposition}{Proposition}[section]
\numberwithin{equation}{section}

\let\Re\relax
\DeclareMathOperator{\Re}{Re}
\let\Im\relax
\DeclareMathOperator{\Im}{Im}
\begin{document}
	
\title{Operator inequalities via the triangle inequality}
\author{Shigeru Furuichi, Mohammad Sababheh, and Hamid Reza Moradi}
\maketitle

\begin{abstract}
This article improves the triangle inequality for complex numbers, using the Hermite-Hadamard inequality for convex functions. Then, applications of the obtained refinement are presented to include some operator inequalities. The operator applications include numerical radius inequalities and operator mean inequalities.
\end{abstract}
\pagestyle{myheadings}
\markboth{\centerline {}}
{\centerline {}}
\bigskip
\bigskip
\section{Main Results}
In the field of Mathematical inequalities, interest in refining existing inequalities or sharpening them has been at the center of researchers' attention; see \cite{FM2020} for example.

One of the most powerful tools in obtaining new inequalities or sharpening existing ones is the use of convex functions. Recall that a function $f:J\to\mathbb{R}$ is said to be convex on the interval $J$ if it satisfies the basic inequality
\begin{align}\label{ineq_intro_basic}
f((1-t)a+tb)\leq (1-t)a+tb,
\end{align}
for all $a,b\in J$ and $0\leq t\leq 1.$ This inequality was refined  by the form
\begin{align}\label{ineq_intro_ref}
f((1-t)a+tb)+2r_t\left(\frac{f(a)+f(b)}{2}-f\left(\frac{a+b}{2}\right)\right)\leq (1-t)f(a)+tf(b),
\end{align}
where $r_t=\min\{t,1-t\}.$ Applications of this inequality can be found in \cite{mitroi,sab_mia,sab_mjom}.

One of the most useful inequalities in convex analysis is the so called Hermite-Hadamard inequality, which states
\begin{align}\label{ineq_intro_HH}
f\left(\frac{a+b}{2}\right)\leq \frac{1}{b-a}\int_{a}^{b}f(t)dt\leq \frac{f(a)+f(b)}{2};
\end{align}
as a refinement of \eqref{ineq_intro_basic} when $t=\frac{1}{2}.$

In this paper we employ the Inequality \eqref{ineq_intro_HH} to refine the well known triangle inequality 
$$|c+d|\leq |c|+|d|, c,d\in \mathbb{C}.$$
More precisely, we show that 
\begin{equation}\label{intro_ineq_ref_tri}
	\left|c+d \right|\le 2\int\limits_{0}^{1}{\left| sc+(1-s)d \right|ds}\le \left| c \right|+\left| d \right|.
	\end{equation}
This will enable us to present a refinement of the Cauchy-Schwartz inequality for the inner product. Then applications that include refined forms for some numerical radius inequalities and operator mean inequalities will be given.

For this, we need to recall some notions related to Hilbert space operators. Let $\mathcal{H}$ be a complex Hilbert space with inner product $\left<\cdot,\cdot\right>:\mathcal{H}\times\mathcal{H}\to\mathbb{C}$ and let $\mathcal{B}(\mathcal{H})$ be the $C^*$-algebra of all bounded linear operators on $\mathcal{H}$. For $A\in\mathcal{B}(\mathcal{H})$, the operator norm and the numerical radius are defined respectively as
$$\|A\|=\sup_{\|x\|=1}\|Ax\|\;{\and}\;w(A)=\sup_{\|x\|=1}\left|\left<Ax,x\right>\right|.$$
It is well known that
\begin{align*}
\frac{1}{2}\|A\|\leq w(A)\leq \|A\|.
\end{align*}
Refining these inequalities has occupied an adequate area of research in this field. For example, in \cite{1} Kittaneh showed that
\begin{align}\label{intro_ineq_w_half}
w(A)\leq \frac{1}{2}\left\|\;|A|+|A^*|\;\right\|,
\end{align}
where $A^*$ is adjoint operator of $A$. The fact that this refines the inequality $w(A)\leq \|A\|$ is due to the triangle inequality and the observation $\|\;|A^*|\;\|=\|\;|A|\;\|.$ 

Using \eqref{intro_ineq_ref_tri} we will be able to present a refined form of \eqref{intro_ineq_w_half}, where we find a scalar $\alpha$ such that $\frac{1}{2}\leq\alpha\leq 1$ and
$$w(A)\leq\frac{\alpha}{2}\left\|\;|A|+|A^*|\;\right\|,$$  
for a certain class of operators. We should remark that finding better bounds for the numerical radius has received a renowned interest in the last few years, as one can see in \cite{bakh_sheb,bhu,bou,gau,haj,Hirzallah,1,Kittaneh 2,mosl,Sababheh}.

Using the same approach, we will be able to find an inequality that relates the geometric mean of $|A|^{2v}$ and $|A|^{2(1-v)}$ with the numerical radius. For this, we recall that the weighted geometric mean of the two positive definite operators $A,B$ is \cite{ando_hiai} $$A\sharp_t B=A^{\frac{1}{2}}\left(A^{-\frac{1}{2}}BA^{-\frac{1}{2}}\right)^tA^{\frac{1}{2}}; 0\leq t\leq 1.$$ When $t=\frac{1}{2},$ we write $\sharp$ instead of $\sharp_{\frac{1}{2}}.$

Our results will make use of the angle $\angle_{x,y}$ between two vectors $x,y\in \mathbb{C}^n$ or $x,y\in\mathcal{H}.$ For such $x,y$ the Cauchy-Schwartz inequality states that $|\left<x,y\right>|\leq \|x\|\;\|y\|.$ From this, the angle between $x,y$ can be defined by
$$\angle_{x,y}=\cos^{-1}\left(\frac{|\left<x,y\right>|}{\|x\|\;\|y\|}\right).$$
It is implicitly understood that $0\leq \angle_{x,y}\leq\pi.$

\section{Main results}
In this section, we present our main results. We begin by showing the scalar inequalities, leading to operator inequalities.
\subsection{Scalar inequalities}
Our main result in this section is  the following refinement of the triangle inequality.
\begin{theorem}\label{1}
Let $c$ and $d$ be two complex numbers. Then
	\begin{equation*}
	\left| \frac{c+d}{2} \right|\le \int\limits_{0}^{1}{\left| sc+(1-s)d \right|ds}\le \dfrac{\left| c \right|+\left| d \right|}{2}.
	\end{equation*}
\end{theorem}
\begin{proof}
Let $a,b\in\mathbb{C}$ and define the function $f:\mathbb{R}\to[0,\infty)$ by $f(t)=|a+tb|.$ We will show that $f$ is convex on $\mathbb{R}$. For this, assume first that $\frac{a}{b}=\alpha\in\mathbb{R}$. Then $f(t)=|a|\;|1+\alpha t|$, which is convex; being the absolute value function. Now, if $\frac{a}{b}\not\in\mathbb{R}$, then $f$ is twice differentiable on $\mathbb{R}$ and
$$
f''(t)=\frac{\left(\Re(a)  \cdot \Im(b)-\Im (a) \cdot \Re(b)\right)^2}{\left\{\left(\mathcal \Re(a)+t \cdot\Re(b)\right)^2+\left(\Im(a)+t \cdot\Im(b)\right)^2\right\}^{3/2}}\ge 0.
$$
 Thus, $f$ is convex on $\mathbb{R}$. By the Hermite-Hadamard inequality \eqref{ineq_intro_HH}, we have
	\[\left| a+\left( \frac{x+y}{2} \right)b \right|\le \frac{1}{x-y}\int\limits_{y}^{x}{\left| a+tb \right|dt}\le \frac{\left| a+xb \right|+\left| a+yb \right|}{2},\]
for $x>y$. Putting $c=a+xb$ and $d=a+yb$, we  have $a=\dfrac{dx-cy}{x-y}$ and $b=\dfrac{c-d}{x-y}$. This implies
	$$
	\left| \frac{c+d}{2} \right|\le \frac{1}{ x-y }\int\limits_{y}^{x}{\left| \frac{dx-cy+(c-d)t}{x-y} \right|dt}\le \frac{\left| c \right|+\left| d \right|}{2}.
	$$
Now, if we set $\dfrac{dx-cy+(c-d)t}{x-y}=sc+(1-s)d$, then we have $\dfrac{1}{x-y}dt=ds$. So for complex numbers $c$ and $d$,
	\begin{equation*}
	\left| \frac{c+d}{2} \right|\le \int\limits_{0}^{1}{\left| sc+(1-s)d \right|ds}\le \dfrac{\left| c \right|+\left| d \right|}{2},
	\end{equation*}
as desired.
\end{proof}

As a direct consequence of Theorem \ref{1}, we can improve the celebrated Cauchy-Schwartz inequality.

\begin{corollary}\label{2}
Let $x,y\in \mathcal H$. Then
	\[\left| \left\langle x,y \right\rangle  \right|\le \int\limits_{0}^{1}{\left| t{{e}^{i\theta }}+\left( 1-t \right){{e}^{-i\theta }} \right|dt}\left\| x \right\|\left\| y \right\|\le \left\| x \right\|\left\| y \right\|,\]
where $\theta ={{\angle }_{x,y}}$.
\end{corollary}
\begin{proof}
We have
\begin{equation}\label{6}
\left| \left\langle x,y \right\rangle  \right|=\left| \cos \theta  \right|\left\| x \right\|\left\| y \right\|.
\end{equation}
Since 
	\[\cos \theta =\operatorname{\mathcal R}\left( {{e}^{i\theta }} \right)=\frac{{{e}^{i\theta }}+{{e}^{-i\theta }}}{2},\]
Theorem \ref{1} applied for $c=e^{i\theta}$ and $d=e^{-i\theta}$ implies that
	\[\left| \left\langle x,y \right\rangle  \right|\le \int\limits_{0}^{1}{\left| t{{e}^{i\theta }}+\left( 1-t \right){{e}^{-i\theta }} \right|dt}\left\| x \right\|\left\| y \right\|\le \left\| x \right\|\left\| y \right\|,\]
	as desired.
\end{proof}

\begin{remark}
The inequalities in Corollary \ref{2} mean
$$
\vert \cos \theta \vert \le \int\limits_{0}^{1}{\left| t{{e}^{i\theta }}+\left( 1-t \right){{e}^{-i\theta }} \right|dt} \le 1.
$$
\end{remark}

To calculate the constant that appears in Corollary \ref{2}, notice that for an arbitrary $\theta\in\mathbb{R}$,
\[\begin{aligned}
   \left| t{{e}^{i\theta }}+\left( 1-t \right){{e}^{-i\theta }} \right|&=\left| t\left( \cos \theta +i\sin \theta  \right)+\left( 1-t \right)\left( \cos \theta -i\sin \theta  \right) \right| \\ 
 & =\left| \cos \theta +i\left( 2t-1 \right)\sin \theta  \right| \\ 
 & =\sqrt{{{\cos }^{2}}\theta +{{\left( 2t-1 \right)}^{2}}{{\sin }^{2}}\theta }.  
\end{aligned}\]

For the case $\sin\theta =0$, we have $\int\limits_{0}^1\sqrt{\cos^2\theta}dt =|\cos\theta|$.
If we assume $\sin\theta \neq 0$, then
\[\begin{aligned}
 \int_{0}^{1}{\sqrt{{{\cos }^{2}}\theta +{{\left( 2t-1 \right)}^{2}}{{\sin }^{2}}\theta }dt}
&=\frac{|\sin \theta|}{2}\int_{-1}^1\sqrt{s^2+\cot^2\theta}ds\\
&=\frac{|\sin\theta|}{4}\left[s\sqrt{s^2+\cot^2\theta}+\cot^2\theta \log \left| s+\sqrt{s^2+\cot^2\theta} \right| \right]_{-1}^1\\
&=\frac{1}{2}+\frac{1}{4}|\sin\theta|\cot^2\theta \log \left| \frac{1+|\sin\theta|}{1-|\sin\theta|} \right|=:\mu (\theta).
\end{aligned}\]
Since $\sin\theta > 0$ for $n\pi < \theta < (n+1)\pi$ and $\sin\theta < 0$ for $(n+1)< \pi \theta < (n+2)\pi$, where $n=0,1,2,\cdots$, 
\begin{eqnarray*}
&&\mu(\theta)=\frac{1}{2}+\frac{1}{4}\sin\theta \cot\theta \log \left| \frac{1+\sin\theta}{1-\sin\theta} \right|,\quad n\pi < \theta < (n+1)\pi,\\
&&\mu(\theta)=\frac{1}{2}-\frac{1}{4}\sin\theta \cot\theta \log \left| \frac{1-\sin\theta}{1+\sin\theta} \right|,\quad (n+1)\pi < \theta < (n+2)\pi.
\end{eqnarray*}
Thus we have
$$\mu(\theta)=\frac{1}{4}\left(2+\cos \theta \cot \theta \log \frac{1+\sin \theta}{1-\sin \theta}\right),\quad \theta \neq n\pi,\quad \text{where} \quad n=0,1,2,\cdots,
$$
since $1\pm\sin\theta > 0$ for all $\theta\neq n\pi$, where $n=0,1,2,\cdots$. Since we have $\mu(\theta)\to 1$ when $\theta \to n\pi$ and $|\cos\theta |=1$ for $\theta = n \pi$, where $n=0,1,2,\cdots$, we have
\begin{equation*}
\mu(\theta)=\frac{1}{4}\left(2+\cos \theta \cot \theta \log \frac{1+\sin \theta}{1-\sin \theta}\right),\quad \theta \ge 0.
\end{equation*}

We study the properties of the function $\mu(\theta)$ for $\theta \ge 0$.
It is sufficient to consider $0\le \theta <  \pi$, since $\mu(\theta+\pi)=\mu(\theta)$.
For this purpose, we prepare the following lemma.
\begin{lemma}\label{lemma1.1}
If $0\le x<1$, then 
\begin{equation}\label{lemma1.1_eq01}
\dfrac{2x}{x^2+1}\le \log\dfrac{1+x}{1-x}.
\end{equation}
If $-1< x \le 0$, then  the reversed inequality holds.
\end{lemma}
\begin{proof}
We firstly prove the first statement.
Putting $t:=\dfrac{1+x}{1-x}$ for $0\le x<1$, the inequality \eqref{lemma1.1_eq01} is equivalent to the inequality $\dfrac{t^2-1}{t^2+1}\le \log t$ for $t\ge 1$. Letting $a:=t^2\ge 1$ and using the inequality $\dfrac{a-1}{\log a}\le \dfrac{a+1}{2}$ for $a>0$, we get the desired inequality. This proves the first inequality. 

To prove the reversed inequality for $-1<x\leq 0,$ we set $s:=\dfrac{1}{t}$. Then we have
$\dfrac{s^2-1}{s^2+1}\ge \log s$ for $0<s\le 1$. This inequality is equivalent to the reversed inequality of \eqref{lemma1.1_eq01} for $-1<x\le 0$ setting $s:=\dfrac{1+x}{1-x}$.
\end{proof}
Now we present the monotonicity of the function $\mu(\theta).$ As we mentioned earlier, this function has period $\pi,$ so we study it only on the interval $[0,\pi].$
\begin{proposition}\label{prop1.1}
The function $\mu(\theta)$ is decreasing on the interval $\left[0,\frac{\pi}{2}\right]$ and is increasing in $\left[\frac{\pi}{2},\pi\right].$ 
\end{proposition}
\begin{proof}
By elementary calculations, we have
$$
\mu'(\theta)=\frac{\cos\theta}{8\sin^2\theta}\nu(\theta),\quad where\;\nu(\theta):=4\sin\theta-2(\sin^2\theta +1)\log\frac{1+\sin\theta}{1-\sin\theta}.
$$

Since we have $\lim\limits_{\theta\to n\pi/2}\mu'(\theta)=0$ for $n=0,1,2$, we consider the values $\theta \neq n\pi/2$, where $n=0,1,2$.
Putting $x:=\sin\theta$ for $0< \theta < \pi$, we consider the function $\hat{\nu}(x)=4x-2(x^2+1)\log\dfrac{1+x}{1-x}$  for $0<x<1$. We have $\hat{\nu}(x)<0$ by Lemma \ref{lemma1.1}. Therefore we have $\nu(\theta) \le 0$ for $0\le \theta \le \pi$.

Taking account that $\cos\theta$ is positive when $0<\theta <\dfrac{\pi}{2}$ and is negative when $\dfrac{\pi}{2}<\theta <\pi$, we have
$\mu'(\theta) \le 0$ when $0\le \theta \le \dfrac{\pi}{2}$ and $\mu'(\theta) \ge 0$ when $\dfrac{\pi}{2}\le \theta \le \pi$. This completes the proof.
\end{proof}
\begin{corollary}
The  inequality $\dfrac{1}{2}\le \mu(\theta)\le 1$ holds for $\theta\geq 0.$
\end{corollary}
\begin{proof}
It suffices to consider the values $0\leq \theta\leq \pi.$ We have $\lim\limits_{\theta\to 0}\mu(\theta)=\lim\limits_{\theta\to \pi}\mu(\theta)=1$ and $\lim\limits_{\theta\to \pi/2}\mu(\theta)=\dfrac{1}{2}$. By  Proposition \ref{prop1.1}, we infer that $\dfrac{1}{2}\le \mu(\theta)\le 1$.
\end{proof}

\subsection{Hilbert space Operator inequalities}
In the following theorem, we improve the mixed Schwarz inequality, which states that \cite{kato}
\[\left| \left\langle Ax,y \right\rangle  \right|\le \sqrt{\left\langle {{\left| A \right|}^{2v}}x,x \right\rangle \left\langle {{\left| {{A}^{*}} \right|}^{2\left( 1-v \right)}}y,y \right\rangle },x,y\in\mathcal{H}, 0\leq v\leq 1.\]
This inequality has been used extensively in the literature when dealing with numerical radius inequalities, see \cite{sab_viet} for example.
\begin{theorem}\label{3}
Let $A\in \mathcal B\left( \mathcal H \right)$ with the the polar decomposition $A=U\left| A \right|$ and let $x,y\in \mathcal H$. Then for any $0\le v\le 1$,
\[\left| \left\langle Ax,y \right\rangle  \right|\le \mu \left( \theta  \right)\sqrt{\left\langle {{\left| A \right|}^{2v}}x,x \right\rangle \left\langle {{\left| {{A}^{*}} \right|}^{2\left( 1-v \right)}}y,y \right\rangle },\]
where $\theta ={{\angle }_{{{\left| A \right|}^{v}}x,{{\left| A \right|}^{1-v}}{{U}^{*}}y}}$.
\end{theorem}
\begin{proof}
According to the assumptions and by employing Corollary \ref{2}, we have
\[\begin{aligned}
   \left| \left\langle Ax,y \right\rangle  \right|&=\left| \left\langle U\left| A \right|x,y \right\rangle  \right| \\ 
 & =\left| \left\langle U{{\left| A \right|}^{1-v}}{{\left| A \right|}^{v}}x,y \right\rangle  \right| \\ 
 & =\left| \left\langle {{\left| A \right|}^{v}}x,{{\left| A \right|}^{1-v}}{{U}^{*}}y \right\rangle  \right| \\ 
 & \le \mu \left( \theta  \right)\left\| {{\left| A \right|}^{v}}x \right\|\left\| {{\left| A \right|}^{1-v}}{{U}^{*}}y \right\| \\ 
 & =\mu \left( \theta  \right)\sqrt{\left\langle {{\left| A \right|}^{2v}}x,x \right\rangle \left\langle U{{\left| A \right|}^{2\left( 1-v \right)}}{{U}^{*}}y,y \right\rangle } \\ 
 & =\mu \left( \theta  \right)\sqrt{\left\langle {{\left| A \right|}^{2v}}x,x \right\rangle \left\langle {{\left| {{A}^{*}} \right|}^{2\left( 1-v \right)}}y,y \right\rangle }  
\end{aligned}\]
as desired.
\end{proof}

The following is a straightforward consequence from Proposition \ref{prop1.1}.
\begin{corollary}\label{4} The following holds.
\begin{itemize}
\item[(i)] If $0\le {{\theta }_{1}}<\theta <{{\theta }_{2}}\le \frac{\pi }{2}$, then 
$$\mu \left( \theta  \right)\le \mu \left( {{\theta }_{1}} \right).$$
\item[(ii)] If $\frac{\pi }{2}\le {{\theta }_{1}}<\theta <{{\theta }_{2}}\le \pi $, then
$$\mu \left( \theta  \right)\le \mu \left( {{\theta }_{2}} \right).$$
\end{itemize}

\end{corollary}

Now we are ready to present a new bound for the numerical radius. This form refines \eqref{intro_ineq_w_half}, when $v=\frac{1}{2}$ for a certain class of operators.
\begin{corollary}\label{5}
Let $A\in \mathcal B\left( \mathcal H \right)$ have the polar decomposition $A=U\left| A \right|$, $0\le v\le 1$ and let $\theta_x={{\angle }_{{{\left| A \right|}^{v}}x,{{\left| A \right|}^{1-v}}{{U}^{*}}x}}$ where $x\in \mathcal H$ with $\left\| x \right\|=1$. If 
\begin{itemize}
\item[(i)] If $0\le {{\theta }_{1}}<\theta_x <{{\theta }_{2}}\le \frac{\pi }{2}$ for all unit vectors $x\in\mathcal{H}$, then 
\[\omega \left( A \right)\le \frac{\mu \left( {{\theta }_{1}} \right)}{2}\left\| {{\left| A \right|}^{2v}}+{{\left| {{A}^{*}} \right|}^{2\left( 1-v \right)}} \right\|.\]
\item[(ii)] If $\frac{\pi }{2}\le {{\theta }_{1}}<\theta <{{\theta }_{2}}\le \pi $ for all unit vectors $x\in\mathcal{H}$, then
\[\omega \left( A \right)\le \frac{\mu \left( {{\theta }_{2}} \right)}{2}\left\| {{\left| A \right|}^{2v}}+{{\left| {{A}^{*}} \right|}^{2\left( 1-v \right)}} \right\|.\]
\end{itemize}
\end{corollary}
\begin{proof}
We prove the first inequality. Let $x\in \mathcal H$ be a unit vector. By Theorem \ref{3},
\[\begin{aligned}
   \left| \left\langle Ax,x \right\rangle  \right|&\le \mu \left( \theta_x  \right)\sqrt{\left\langle {{\left| A \right|}^{2v}}x,x \right\rangle \left\langle {{\left| {{A}^{*}} \right|}^{2\left( 1-v \right)}}x,x \right\rangle } \\ 
 & \le \mu \left( {{\theta }_{1}} \right)\sqrt{\left\langle {{\left| A \right|}^{2v}}x,x \right\rangle \left\langle {{\left| {{A}^{*}} \right|}^{2\left( 1-v \right)}}x,x \right\rangle } \\ 
 & \le \mu \left( {{\theta }_{1}} \right)\left( \frac{\left\langle {{\left| A \right|}^{2v}}x,x \right\rangle +\left\langle {{\left| {{A}^{*}} \right|}^{2\left( 1-v \right)}}x,x \right\rangle }{2} \right) \\ 
 & =\frac{\mu \left( {{\theta }_{1}} \right)}{2}\left\langle \left( {{\left| A \right|}^{2v}}+{{\left| {{A}^{*}} \right|}^{2\left( 1-v \right)}} \right)x,x \right\rangle  \\ 
 & \le \frac{\mu \left( {{\theta }_{1}} \right)}{2}\left\| {{\left| A \right|}^{2v}}+{{\left| {{A}^{*}} \right|}^{2\left( 1-v \right)}} \right\|, 
\end{aligned}\]
where the second inequality is obtained from Corollary \ref{4}, and the third inequality follows from the arithmetic-geometric mean inequality. Therefore,
\[\left| \left\langle Ax,x \right\rangle  \right|\le \frac{\mu \left( {{\theta }_{1}} \right)}{2}\left\| {{\left| A \right|}^{2v}}+{{\left| {{A}^{*}} \right|}^{2\left( 1-v \right)}} \right\|.\]
Now, we get the desired result by taking the supremum over all unit vector $x\in \mathcal H$.
\end{proof}

\begin{remark}
Put $v={1}/{2}\;$. If $\theta ={\pi }/{2}\;$, then
\[\omega \left( A \right)=\frac{1}{2}\left\| A \right\|,\]
since
\[\frac{1}{2}\left\| A \right\|\le \omega \left( A \right)\le \frac{1}{4}\left\| \left| A \right|+\left| {{A}^{*}} \right| \right\|\le \frac{1}{2}\left\| A \right\|.\]
\end{remark}

In the following, we present a reverse of the triangle inequality. 
\begin{theorem}\label{7}
Let $c$ and $d$ be two complex numbers. Then for any $0\le t \le 1$,
\[\frac{\left| c \right|+\left| d \right|}{2}-\frac{1}{2r_t }\left( \left( 1-t  \right)\left| c \right|+t \left| d \right|-\left| \left( 1-t  \right)c+t d \right| \right)\le \left| \frac{c+d}{2} \right|,\]
where $r_t =\min \left\{ t ,1-t  \right\}$.
\end{theorem}
\begin{proof}
We know that if $f$ is a convex function on  $ \mathbb{R}$, then for $x,y\in\mathbb{R}$ and any $0\le t \le 1,$
\[f\left( \left( 1-t  \right)x+t y \right)\le \left( 1-t  \right)f\left( x \right)+t f\left( y \right)-2r_t \left( \frac{f\left( x \right)+f\left( y \right)}{2}-f\left( \frac{x+y}{2} \right) \right),\]
where $r_t =\min \left\{ t ,1-t  \right\}$, see \eqref{ineq_intro_ref}. Since for $a,b\in \mathbb{C}$, $f\left( t \right)=\left| a+tb \right|$ is convex  on $\mathbb{R}$, we get 
\[\left| a+\left( \left( 1-t  \right)x+t y \right)b \right|\le \left( 1-t  \right)\left| a+xb \right|+t \left| a+yb \right|-2r_t \left( \frac{\left| a+xb \right|+\left| a+yb \right|}{2}-\left| a+\frac{x+y}{2}b \right| \right).\]
Now, applying the same method as in Theorem \ref{1}, we infer 
\[\left| \left( 1-t  \right)c+t d \right|\le \left( 1-t  \right)\left| c \right|+t \left| d \right|-2r_t \left( \frac{\left| c \right|+\left| d \right|}{2}-\left| \frac{c+d}{2} \right| \right),\] 
as desired
\end{proof}

\begin{remark}
Let $x,y\in \mathcal H$, and let  $r_t=\min \left\{ t ,1-t  \right\}$ with $0\le t \le 1$. By Theorem \ref{7},
\[\begin{aligned}
   0&\le 1-\frac{1}{2r_t}\left( 1-\left| t{{e}^{i\theta }}+\left( 1-t \right){{e}^{-i\theta }} \right| \right)\\ 
 & \le \left| \frac{{{e}^{i\theta }}+{{e}^{-i\theta }}}{2} \right|\\
 &=\left| \cos \theta  \right|, 
\end{aligned}\]
where $\theta ={{\angle }_{x,y}}$. Thus we have
\begin{equation}\label{9}
0\le \gamma_t(\theta) \left\| x \right\|\left\| y \right\|\le \left| \left\langle x,y \right\rangle  \right|,
\end{equation}
where
$$
\gamma_t(\theta):=1-\frac{1}{2r_t}\left( 1-\left| t{{e}^{i\theta }}+\left( 1-t \right){{e}^{-i\theta }} \right| \right)=1-\frac{1}{2r_t}\left( 1-\sqrt{{{\cos }^{2}}\theta +\left( 2t-1 \right)^2{{\sin }^{2}}\theta } \right)
$$
for $\theta \in [0,\pi]$ and $r_t=\min \left\{ t ,1-t  \right\}$ with $0\le t \le 1$. This provides a reverse of the Cauchy-Schwarz inequality.

We notice that $\gamma_{t}(\theta)=\gamma_{1-t}(\theta)$ for $0\leq t\leq 1.$ Also, since we have the Cauchy-Schwartz inequality $|\left<x,y\right>|\leq \|x\|\;\|y\|$, we can use \eqref{9} to obtain  sufficient conditions on the equality $|\left<x,y\right>|= \|x\|\;\|y\|$, as follows.
We notice that $\gamma_{t}(\theta)=1$ when $\theta=0,\pi.$ Thus, when $\angle_{x,y}=0,\pi$, we have $|\left<x,y\right>|= \|x\|\;\|y\|$.\\
Further when $\theta=\frac{\pi}{2}$, $\gamma_{t}(\theta)=0.$

Also, we notice that $\max_{t}\gamma_t(\theta)=\cos\theta,$ which is evident because $\left<x,y\right>=\|x\|\;\|y\|\cos\theta.$

\end{remark}

Again, since $\gamma_t(\theta)$ is periodic in $\theta$ with period $\pi,$ it is sufficient to study it on the interval $[0,\pi].$
\begin{proposition}\label{8}
If $0\le \theta \le \dfrac{\pi}{2}$, then ${{\gamma }_{t}}\left( \theta  \right)$ is decreasing. If $\dfrac{\pi}{2}\le \theta \le \pi$, then ${{\gamma }_{t}}\left( \theta  \right)$ is increasing.
In addition, we have $0\le \gamma_t(\theta)\le 1$.
\end{proposition}
\begin{proof}
Elementary calculations show that
$$
\dfrac{d\gamma_t(\theta)}{d\theta}=-\frac{t(1-t)\sin 2\theta}{r_t\sqrt{\cos^2\theta +(2t-1)^2 \sin^2\theta}}.
$$
This shows the assertion of monotonicity.
Since we have $\lim\limits_{\theta\to n\pi/2}\gamma_t(\theta)=1$ for $n=0,2$ and
$\lim\limits_{\theta\to n\pi/2}\gamma_t(\theta)=\dfrac{2r_t-1+\sqrt{(2t-1)^2}}{2r_t}=0$ for $n=1$, we have  $0\le \gamma_t(\theta)\le 1$ by monotonicity of $\gamma_t(\theta).$
\end{proof}
Now we present an inequality that gives a relation between the geometric mean and the numerical radius of the operator $A$, noting that $\max_{t}\gamma_t(\theta)=\cos\theta.$ 
\begin{corollary}
Let $A\in \mathcal B\left( \mathcal H \right)$ have the polar decomposition $A=U\left| A \right|$, $0\le v\le 1$ and let $\theta_x={{\angle }_{{{\left| A \right|}^{v}}x,{{\left| A \right|}^{1-v}}{{U}^{*}}x}}$ where $x\in \mathcal H$ with $\left\| x \right\|=1$.\begin{itemize}
\item[(i)] If $0\le {{\theta }_{1}}<\theta_x <{{\theta }_{2}}\le \frac{\pi }{2}$, then 
\[\cos\left( {{\theta }_{2}} \right)\left\| {{\left| A \right|}^{2v}}\sharp{{\left| {{A}^{*}} \right|}^{2\left( 1-v \right)}} \right\|\le \omega \left( A \right).\]
\item[(ii)] If $\frac{\pi }{2}\le {{\theta }_{1}}<\theta_x <{{\theta }_{2}}\le \pi $, then
\[\cos\left( {{\theta }_{1}} \right)\left\| {{\left| A \right|}^{2v}}\sharp{{\left| {{A}^{*}} \right|}^{2\left( 1-v \right)}} \right\|\le \omega \left( A \right).\]
\end{itemize}
\end{corollary}
\begin{proof}
We prove case (i) since case (ii) can be proved similarly. In this case, by Proposition \ref{8}, we have ${{\gamma }_{t}}\left( {{\theta }_{2}} \right)\le {{\gamma }_{t}}\left( \theta  \right)$. Putting  $x={{\left| A \right|}^{1-t}}x$, $y={{\left| A \right|}^{1-v}}{{U}^{*}}x$, where $x\in \mathcal H$ is a unit vector, in the inequality \eqref{9}, we have
\[\begin{aligned}
  \cos\left( {{\theta }_{2}} \right)\left\langle {{\left| A \right|}^{2v}}\sharp{{\left| {{A}^{*}} \right|}^{2\left( 1-v \right)}}x,x \right\rangle& \le \cos\left( {{\theta }_{2}} \right)\sqrt{\left\langle {{\left| A \right|}^{2v}}x,x \right\rangle \left\langle {{\left| {{A}^{*}} \right|}^{2\left( 1-v \right)}}x,x \right\rangle } \\ 
 & =\cos\left( {{\theta }_{2}} \right)\sqrt{\left\langle {{\left| A \right|}^{2v}}x,x \right\rangle \left\langle U{{\left| A \right|}^{2\left( 1-v \right)}}{{U}^{*}}x,x \right\rangle } \\ 
 & =\cos\left( {{\theta }_{2}} \right)\left\| {{\left| A \right|}^{v}}x,x \right\|\left\| {{\left| A \right|}^{1-v}}{{U}^{*}}x,x \right\| \\ 
 & \le \left| \left\langle {{\left| A \right|}^{v}}x,{{\left| A \right|}^{1-v}}{{U}^{*}}x \right\rangle  \right| \\ 
 & =\left| \left\langle Ax,x \right\rangle  \right|.  
\end{aligned}\]
Notice that the first inequality follows from the following fact for positive operators $A,B$,
\[\left\langle A\sharp Bx,x \right\rangle \le \sqrt{\left\langle Ax,x \right\rangle \left\langle Bx,x \right\rangle }.\]
Thus,
\[\cos\left( {{\theta }_{2}} \right)\left\langle {{\left| A \right|}^{2v}}\sharp{{\left| {{A}^{*}} \right|}^{2\left( 1-v \right)}}x,x \right\rangle \le \left| \left\langle Ax,x \right\rangle  \right|,\]
and this implies
\[\cos\left( {{\theta }_{2}} \right)\left\| {{\left| A \right|}^{2v}}\sharp{{\left| {{A}^{*}} \right|}^{2\left( 1-v \right)}} \right\|\le \omega \left( A \right),\]
as desired.
\end{proof}

\vskip 0.5 true cm

{\tiny (S. Furuichi) Department of Information Science, College of Humanities and Sciences, Nihon University, Setagaya-ku, Tokyo, Japan}

{\tiny \textit{E-mail address:} zahraheydarbeygi525@gmail.com}

\vskip 0.3 true cm

{\tiny (M. Sababheh) Department of Basic Sciences, Princess Sumaya University for Technology, Amman, Jordan}
	
	{\tiny\textit{E-mail address:} sababheh@psut.edu.jo}

\vskip 0.3 true cm

{\tiny (H. R. Moradi) Department of Mathematics, Payame Noor University (PNU), P.O. Box, 19395-4697, Tehran, Iran
	
	\textit{E-mail address:} hrmoradi@mshdiau.ac.ir}
\end{document}